\documentclass{amsart}
\usepackage{amsthm}
\usepackage[english]{babel}
\usepackage{amssymb}
\usepackage{epsfig}
\usepackage{graphicx, color}
\usepackage{hyperref}
\usepackage{natbib}
\let\cite=\citet

\newcommand{\NN}{{\mathbb N}}

\newcommand{\EE}{{\mathbb E}}

\newcommand{\PP}{{\mathbb P}}

\renewcommand{\geq}{\geqslant}
\renewcommand{\le}{\leqslant}

\newcommand\1{\leavevmode\hbox{\rm \small1\kern-0.35em\normalsize1}}

\def\A{{\mathcal A}}
\def\B{{\mathcal B}}
\def\C{{\mathcal C}}
\def\D{{\mathcal D}}
\def\E{{\mathcal E}}
\def\F{{\mathcal F}}

\def\M{{\mathcal M}}

\makeatletter
\def\blfootnote{\xdef\@thefnmark{}\@footnotetext}
\makeatother

\numberwithin{equation}{section}

\theoremstyle{plain}

\newtheorem{Theo}{Theorem}
\newtheorem{Ex}{Example}
\newtheorem*{Exe}{Example \ref{exITA}}
\newtheorem*{Exem}{Example \ref{ex:3lettres}}
\newtheorem{Prop}{Proposition}
\newtheorem{Cor}{Corollary}
\newtheorem{Lem}{Lemma}

\newtheorem{Def}{Definition}
\newtheorem{Rem}{Remark}

\begin{document}

\title{On non-regular $g$-measures}
\author{Sandro Gallo}
\author{Fr\'ed\'eric Paccaut}
\begin{abstract}
We prove that $g$-functions whose set of discontinuity points has strictly negative topological pressure and which satisfy an assumption that is weaker than non-nullness, have at least one stationary $g$-measure. We also obtain uniqueness by adding conditions on the set of continuity points.
\end{abstract}

\maketitle\blfootnote{Both authors were partially supported by CAPES grant AUXPE-PAE-598/2011}

%%%%%%%%%%%%%%%%%%%%%%%%%%%%%%%%%%%%%%%%%%%%%%%%%%%%%%%%%%%%%%%%%%%%%%%%%%%%%%%%%%%%%%%%%%%%%%%%%%%%%%%%%%%
\section{Introduction}
%%%%%%%%%%%%%%%%%%%%%%%%%%%%%%%%%%%%%%%%%%%%%%%%%%%%%%%%%%%%%%%%%%%%%%%%%%%%%%%%%%%%%%%%%%%%%%%%%%%%%%%%%%%

The $g$-measures on $A^{\mathbb{Z}}$ ($A$ discrete) are the measures for which the conditional probability of one state at any time, given the past, is specified by a function $g$, called $g$-function. In this paper, $g$-measures will always refer to \emph{stationary} measures. The main  question we answer in the present paper is the following:  what conditions on $g$-functions $g$ will ensure the existence of a (stationary) $g$-measure? 

It is well-known that the continuity of $g$ implies existence if the alphabet $A$ is finite. Here we extend this result to discontinuous $g$-functions by proving that existence holds whenever the topological pressure of the set of discontinuities of $g$ is strictly negative, even when $g$ is not necessarily non-null.

\vspace{0.3cm}

The name $g$-measure was introduced by \cite{Keane} in Ergodic Theory to refer to an extension of the Markov measures, in the sense that the function $g$ may depend on a unbounded part of the past. In the literature of stochastic processes, these objects already existed under the names ``Cha\^ines \`a liaison compl\`ete'' or ``chains of infinite order'', respectively coined by \cite{doeblin/fortet/1937} and \cite{harris/1955}. The function $g$ is also called set of transition probabilities, or probability kernel. Given a function $g$ (or \emph{probability kernel}), the most basic questions are the following: does it specify a $g$-measure (or stationary \emph{stochastic process})? If yes, is it unique? 
To answer these questions, the literature mainly focussed on the continuity  assumption for $g$ (see \cite{onicescu/mihoc/1935, doeblin/fortet/1937, harris/1955, Keane, Ledrappier, JO, fernandez/maillard/2005} and many other). This assumption gives ``for free'' the existence of the $g$-measure. For this reason, uniqueness and the study of the statistical properties of the resulting unique measure have been the centre of the attention from the beginning of the literature.
Only recently, \cite{gallo/2009, CCPP, desantis/piccioni/2012} studied $g$-measures with functions $g$ that were not necessarily continuous. However, no general criteria has been given regarding the existence issue, either because these works are example-based, or because the obtained conditions  are restrictive, implying both existence and uniqueness. This rises a natural motivation for  finding a general criteria for the existence of $g$-measures.

\vspace{0.3cm}

A second motivation is the analogy with one-dimensional Gibbs measures. In statistical mechanics, the function specifying the conditional probabilities  with respect to both \emph{past and future} is called a specification. The theorem of \cite{kozlov/1974} states that Gibbs measures have continuous and strictly positive specifications. 
 Stationary measures having support on the set of points where the specification is continuous are called almost-Gibbsian (\cite{maes/redig/vanMoffaert/leuven/1999}). Clearly, Gibbsian measures are almost-Gibbsian. \cite{FGM} proved that regular $g$-measures (associated to continuous and strictly positive function $g$) might not be Gibbs measures, still they are always almost-Gibbsian. Thus, although the nomenclature of Gibbsianity cannot be imported directly to the case of $g$-measures, it is tempting to try to find ``almost-regular" $g$-measures. 
 
 \vspace{0.3cm}

Going further in the analogy between $g$-measures and (almost-)Gibbs measures, a natural idea is to look for a $g$-measure having support inside the set of continuity points of $g$. Of course, it is not an easy task to control the support of a measure before knowing its existence. An idea is then to put a topological assumption on the set of discontinuity points of $g$, ensuring that this set will have $\mu$-measure $0$, whenever the $g$-measure $\mu$ exists. In the vein of \cite{BPS}, this is done in the present paper by using the topological pressure of the set of discontinuity points of $g$. 
Theorem \ref{existence} states that there exist  $g$-measures when the function $g$ has a set of discontinuity points with negative topological pressure, even without assuming non-nullness. As a corollary (Corollary \ref{coro}), a simple condition on the set of discontinuity points of a function $g$ is given, which may appear more intuitive to the reader not familiar with the concept of topological pressure. The set of discontinuity points of $g$ can be seen as a tree where each branch is $A^{-\mathbb{N}}$. The new condition is that the upper exponential growth rate of this tree is smaller than a constant that depends on $\inf_X g$ (or, if non-nullness is not assumed, on some parameter explicitly computable on $g$). Our last result (Theorem \ref{theo:ExistsUniqueMixing}), based on the work of \cite{JO}, gives explicit sufficient conditions on the set of continuity points of discontinuous kernels $g$ (satisfying our conditions of existence) ensuring uniqueness.

%%%%%%%%%%%%%%%%%%%%%%%%%%%%%%%%%%%%%%%%%%%%%%%%%%%%%%%%%%%%%%%%%%%%%%%%%%%%%%%%%%%%%%%%%%%%%%%%%%%%%%%%%%%
\section{Notations, definitions and main results}
%%%%%%%%%%%%%%%%%%%%%%%%%%%%%%%%%%%%%%%%%%%%%%%%%%%%%%%%%%%%%%%%%%%%%%%%%%%%%%%%%%%%%%%%%%%%%%%%%%%%%%%%%%

Let $(A,\A)$ be a measurable space, where $A$ is a finite set (the alphabet) and $\A$ is the associated discrete $\sigma$-algebra. We will denote by $|A|$ the cardinal of $A$. Define $X=A^{-\NN}$ (we use the convention that $\mathbb{N}=\{0,1,2,\ldots\}$), endowed with the product of discrete topologies and with the $\sigma$-algebra $\F$ generated by the coordinate applications. For any $x\in X$, we will use the notation $x=(x_{-i})_{i\in \NN}=x_{-\infty}^{0}=\ldots x_{-1}x_{0}$. For any $x\in X$ and $z\in X$, we denote, for any $k\ge0$,  $zx_{-k}^{0}=\ldots z_{-2}z_{-1}z_{0}x_{-k}\ldots x_{0}$, the concatenation between $x_{-k}^{0}$ and $z$.
In other words, $zx_{-k}^{0}$ denotes a new sequence $y\in X$ defined by $y_{i}=z_{i+k+1}$ for any $i\leq -k-1$ and $y_{i}=x_{i}$ for any $-k\leq i\leq 0$. Finally, the length of any finite string $v$ of elements of $A$, that is, the number of letters composing the string $v$, will be written $|v|$.\\

Define the shift mapping $T$ as follows :
$$
\begin{array}{cccc}
T: & X & \rightarrow & X \\
\  & (x_n)_{n\le 0} & \mapsto & (x_{n-1})_{n\le 0}.
\end{array}
$$
The mapping $T$ is continuous and has $|A|$ continuous branches called $T_a^{-1}, a\in A$.
Denote by $\M$ the set of Borelian probability measures on $X$, by $\B$ the set of bounded functions and by $\C$ the set of continuous functions.
The characteristic functions will be written $\1$.

A $g$-function is a $\F$-measurable function $g:X\to [0,1]$ such that
\begin{equation}\label{g-function}
\forall x\in X,\,\,\, \sum_{y:T(y)=x}g(y)=\sum_{a\in A}g(xa)=1.
\end{equation}

\begin{Ex}\emph{
Matrix transitions of $k$-steps Markov chains, $k\ge1$, are the simplest example of $g$-functions. They satisfy $g(xa)=g(ya)$ whenever $x_{-k+1}^{0}=y_{-k+1}^{0}$, $\forall a$. }
\end{Ex}
\begin{Ex}\emph{\label{comb} Let us introduce one of the simplest examples of non-Markovian $g$-function on $A=\{0,1\}$. Let $(q_{n})_{n\in \NN\cup\{\infty\}}$ be a sequence of $[0,1]$-valued real numbers. Set $\tilde{g}(x1)=q_{\ell(x)}$ where $\ell(x):=\inf\{k\geq 0:x_{-k}=1\}$ for any $x\in A^{-\mathbb{N}}$ (with the convention that $\ell(x)=\infty$ whenever $x_{-i}=0$ for all $i\le 0$). Notice that the value of $\tilde{g}(x)$ depends on the distance to occurrence of a symbol $1$ in the sequence $\ldots x_{-1}x_{0}$. Therefore, for any $k\geq 1$ the property that $g(xa)=g(ya)$ whenever $x_{-k+1}^{0}=y_{-k+1}^{0}=0_{-k+1}^{0}$ does not hold. This is not the transition matrix of a Markov chain. We will come back to this motivating example several times throughout this paper.}
\end{Ex}

\begin{Def}
An $A$-valued stochastic processes $(\xi_n)$ defined on a probability space $(\Omega,\mathcal{G},\PP)$ is specified by a given $g$-function $g$ if
$$
{\PP}(\xi_0=a|(\xi_k)_{k<0})=g(\ldots \xi_{-2}\xi_{-1}a)\ \ \PP\ \mbox{almost surely}.
$$
The distribution of a \emph{stationary} process $(\xi_n)$ of this form is called a $g$-measure.
\end{Def}
Here is a more ergodic oriented, equivalent definition:

\begin{Def}Let $g$ be a $g$-function. A probability measure $\mu\in\M$ is called a $g$-measure if $\mu$ is $T$-invariant and for $\mu$ almost every $x\in X$ and for every $a\in A$:
$$
{\EE}_{\mu}(\1_{\{x_0=a\}}|\F_1)(x)=g(T(x)a).
$$
with $\F_1=T^{-1}\F$.
\end{Def}

Given a $g$-function, the existence of a corresponding  $g$-measure is not always guaranteed. For instance, coming back to example \ref{comb}, \cite{CCPP} proved that if $\prod_{k\ge1}\sum_{i=0}^{k-1}(1-q_{i})=\infty$ and $q_{\infty}>0$, then there does not exist any  $g$-measures for $\tilde{g}$. Another simple example is given by \cite{Keane} on the torus. In general, a sufficient condition for the existence of a  $g$-measure corresponding to some fixed $g$-function is to assume that $g$ is continuous in every point (see \cite{Keane} for instance). Continuity here is understood  with respect to the discrete topology, that is, $g$ is continuous at the point $x$ if for any $z$, we have
\[
g(zx_{-k}^{0})\stackrel{k\rightarrow\infty}{\longrightarrow} g(x).
\] 
Continuity is nevertheless not necessary for existence, as shown, one more time, by the $g$-function $\tilde{g}$ of example \ref{comb}. For instance, let $q_{i}=\epsilon<1/2$ when $i$ is odd and $q_{i}=1-\epsilon$ when $i$ is even, and put $q_{\infty}>0$. Observe that in this case $\tilde{g}$ has a discontinuity at $0_{-\infty}^{0}$, since $\tilde{g}(1_{-\infty}^{0}0_{-k}^{0})$ oscillates between $\epsilon$ and $1-\epsilon$ when $k$ increases. But it is well-known that $\tilde{g}$ has a  $g$-measure (see \cite{CCPP} or \cite{gallo/2009} for instance). \\

The preceding observations yield to our first issue, which is to give a general condition on the set of discontinuities of $g$, under which there still exists a  $g$-measure.  This is the content of Theorem \ref{existence} which we will state after introducing some further definitions. \\

The cylinders are defined in the usual way by
$$
C_n(x)=\{w\in X, w_{-n+1}^0=x_{-n+1}^0\}\,,\,\,\forall x\in X,
$$
and the set of $n$-cylinders is
$$
\C_n=\{C_n(x), x\in X\}.
$$
Define, for $x\in X$ and $n\in{\NN}$, $n\ge1$
$$
g_n(x)=\prod_{i=0}^{n-1}g(T^i(x)).
$$
The topological pressure of a measurable set $S\subset X$ is defined by
$$
P_g(S)=\limsup_{n\to+\infty}\frac1n\log\sum_{{B\in\C_n}\atop{B\cap S\neq\emptyset}}\sup_Bg_n.
$$
Let $\D$ be the set of discontinuity points of $g$. Let $\C_n(\D)$ be the union of $n$-cylinders that intersect $\D$ :
$$
\C_n(\D)=\bigcup_{x\in\D}C_n(x).
$$
For $n\in\NN$, set $\E_n=T^{-1}T\C_{n+1}(\D)$ (notice that $\E_0=X$ and $\E_{n+1}\subset\E_n$). $\E_n$ is the set of points that write $yx_{-n}^{-1}a$, with $a\in A$, $x_{-\infty}^0\in\D$ and $y\in X$.

\begin{Theo}\label{existence}
Let $g$ be a $g$-function with discontinuity set $\D$. Assume 
$$
\begin{array}{ll}
{\bf (H1)} & {\exists N\in\NN,\,\exists\varepsilon>0,\,\inf_{\E_N}g=\varepsilon},\\
{\bf (H2)} & P_g(\D)<0,
\end{array}
$$
{then there exists at least a $g$-measure and its support is contained in $X\setminus\mathcal{D}$}.
\end{Theo}

\begin{Rem}Hypothesis {\bf (H1)} is strictly weaker than the ``strong non-nullness'' assumption $\inf_{X}g>0$, since the later corresponds to {\bf (H1)} being satisfied for $N=0$ and Example \ref{ex:3lettres} below satisfies {\bf (H1)} and is not strongly non-null.
\end{Rem}

\begin{Rem}\label{discontinuite_fini}
\emph{Notice that {\bf (H2)} is fulfilled for instance when
$\D$ is a finite set and $\inf_X g>0$ (i.e. {\bf (H1)} is fulfilled with $N=0$). This is, in particular, the case of our simplest Example \ref{comb} when the $q_{i}$'s are oscillating between $\varepsilon$ and $1-\varepsilon$. {Notice also that {\bf (H2)} is fulfilled as well as soon as $\D$ is finite, {\bf (H1)} is fulfilled with $N>1$ and $T\D\subset \D$. This will be an easy consequence of Corollary \ref{coro}.}}
\end{Rem}

\begin{Rem}
\emph{Notice also that {\bf (H2)} implies that $g$ cannot be everywhere discontinuous. Namely, the property \ref{g-function} of a $g$-function entails :
$$
\forall n\in{\NN}^*, \forall y\in X, \sum_{x_{-n+1}^{0}\in A^n}g_{n}(yx_{-n+1}^{0})=1
$$
therefore $\sum_{B\in\C_n}\sup_Bg_n\geq 1$ which in turn implies that $P_g(X)\geq 0$.}
\end{Rem}

\begin{Ex}\label{exITA}\emph{
This example was presented in \cite{desantis/piccioni/2012} (see Example 2 therein) on $\{-1,+1\}$. Here we adapt it on the alphabet $A=\{0,1\}$. As $\tilde{g}$, the $g$-function we introduce here has a unique discontinuity point along $0_{-\infty}^{0}$, but the dependence on the past does not stop at the last occurrence of a $1$. 
Recall that $\ell(x):=\inf\{k\geq 0:x_{-k}=1\}$.
Let $g(0_{-\infty}^{0}1)=\epsilon>0$, and for any $x\neq 0_{-\infty}^{0}$ and any $a\in \{0,1\}$ let
\[
g(xa)=\epsilon+(1-2\epsilon)\sum_{n\geq 1}{\bf 1}\{x_{-\ell(x)-n}=a\}q_{n}^{\ell(x)},
\]
where, for any $l\geq 0$, $(q_{n}^{l})_{n\geq 1}$ is a probability distribution on the integers. 
This kernel has a  discontinuity at $0_{-\infty}^{0}$ since for each $k\in{\NN}$,
\begin{align*}
g(\ldots1110_{-k}^{0}1)=\epsilon+(1-2\epsilon)\sum_{n\geq 1}q_{n}^{k+1}= 1-\epsilon\neq \epsilon,
\end{align*}
but it is continuous at any other point, since for any $x$ such that $\ell(x)=l<+\infty$, for any $z$ and $k>l$
\begin{align*}
g(\ldots z_{-1}z_{0}x_{-k}^{0}1)=\epsilon+(1-2\epsilon)\left[\sum_{j=1}^{k-l}{\bf 1}\{x_{-l-j}=1\}q_{j}^{l}+\sum_{j\geq k-l+1}{\bf 1}\{z_{k-l+1-j}=1\}q_{j}^{l}\right]
\end{align*}
which converges to $g(x1)=\epsilon+(1-2\epsilon)\sum_{j\geq 1}{\bf 1}\{1=x_{-l-j}\}q_{j}^{l}$.  Under some assumptions on the set of distributions $((q_{n}^{l})_{n\geq 1})_{l\geq 0}$, \cite{desantis/piccioni/2012} proved existence, uniqueness and perfect simulation while our Theorem \ref{existence} guarantees existence of a $g$-measure, without any further assumptions on this sequence of distributions. 
}
\end{Ex}

Theorem \ref{existence} involves the notion of topological pressure, which is not always easy to extract from the set of discontinuities. We now introduce two simple criteria on the set $\mathcal{D}$ of a $g$-function, that will imply existence.

\begin{Def}For any $n\geq 0$, let us denote $\mathcal{D}^{n}:=\{x_{-n+1}^{0}\}_{x\in\mathcal{D}}$.
The \emph{upper exponential growth rate of $\mathcal{D}$} is
\begin{equation}\label{eq:growth}
\bar{gr}(\mathcal{D}):=\limsup_{n}|\mathcal{D}^{n}|^{1/n}.
\end{equation}
\end{Def}
Although this nomenclature is generally reserved for trees, we use it here as there exists a natural way to represent the set $\mathcal{D}$ as a rooted tree (a subtree of $A^{-\mathbb{N}}$) with the property that each branch, representing an element of $\mathcal{D}$, is infinite, and each node has between $1$ and $|A|$ sons. For instance, in the particular case of $\tilde{g}$ (Example \ref{comb}), the tree is the single branch $0_{-\infty}^{0}$ and $\mathcal{D}^{n}=0_{-n+1}^{0}$.

\begin{Cor}\label{coro}
{Let $g$ be a $g$-function with discontinuity set $\D$. Assume either,
$$
\begin{array}{ll}
{\bf (H1')}&\exists \varepsilon>0, \inf_{X}g=\varepsilon,\\
{\bf (H2')} & \bar{gr}({\mathcal{D}})<[1-(|A|-1)\varepsilon]^{-1}, \\
\end{array}
$$
or
$$
\begin{array}{ll}
{\bf (H1)} & {\exists N\in\NN,\,\exists\epsilon>0,\,\inf_{\E_N}g=\epsilon},\\
{\bf (H2')} & \bar{gr}({\mathcal{D}})<[1-(|A|-1)\varepsilon]^{-1}, \\
{\bf (H3)} & T\D\subset\D,
\end{array}
$$
then there exists at least a $g$-measure and its support is contained in $X\setminus\mathcal{D}$.}
\end{Cor}

Intuitively, Corollary \ref{coro} states that if $\varepsilon$ (which plays the role of a ``non-nullness parameter'' for $g$) is sufficiently large, it may compensate the set of discontinuities of $g$, allowing  $g$-measures to exist, with support on the continuity points. Notice that this assumption allows $\mathcal{D}$ to be uncountable, as shown in the following example.

\begin{Ex}\label{ex:3lettres}\emph{
Let $A=\{0,1,2\}$, and consider the function $\ell$ defined as in Examples \ref{exITA} and \ref{comb}. Let also  $N_{0}$, $N_{1}$ and $N_{2}$ be three disjoint finite subsets of $\NN$. The $g$-function is defined as follows: for $x\in\{0,2\}^{-\mathbb{N}}$, put $
g(x1)=g(x0)=0.3$, for $x$ such that $\ell(x)\in N_{0}\cup N_{1}\cup N_{2}$, put
\begin{equation}
\begin{array}{ccc}
g(x1)=g(x2)=1/2&\textrm{if}\,\, \ell(x)\in N_{0}\\
g(x0)=g(x2)=1/2&\textrm{if} \,\,\ell(x)\in N_{1}\\
g(x0)=g(x1)=1/2&\textrm{if} \,\,\ell(x)\in N_{2},
\end{array}
\end{equation}
and for any $x$ such that $\ell(x)\in \NN\setminus \{N_{0}\cup N_{1}\cup N_{2}\}$, put
$$
g(x1)=g(x0)=0.26+\sum_{k\geq 1}\theta_{k}x_{-\ell(x)-k},
$$
where $(\theta_{i})_{i\geq 1}$ satisfies $\theta_{i}\geq 0$ and $\sum_{i\geq 1}\theta_{i}<0.03$. Observe that, for any $x\in\{0,2\}^{-\NN}$, $g(\ldots111x_{-k}^{0}1)<0.29$ for any sufficiently large $k$, and therefore does not converge to $0.3$. So $\{0,2\}^{-\NN}\subset \mathcal{D}$. On the other hand, any point $x$ satisfying $\ell(x)\in N_{0}\cup N_{1}\cup N_{2}$ is trivially continuous, and any point $x$ satisfying $\ell(x)\in \NN\setminus\{N_{0}\cup N_{1}\cup N_{2}\}$ is continuous since for any $k>l$ and any $y\in \{0,1,2\}^{-\mathbb{N}}$, 
\[
g(\ldots y_{-2}y_{-1}y_{0}x_{-k}^{0}1)=0.26+\sum_{i=1}^{k-l}\theta_{i}x_{-l-i}+\sum_{i\geq k-l+1}\theta_{i}y_{i-k+l-1}
\]
which converges to $0.26+\sum_{i\geq 1}\theta_{i}x_{-l-i}$.
So $\mathcal{D}=\{0,2\}^{-\NN}$ (which is uncountable), $|\mathcal{D}^{n}|=2^{n}$ and consequently $\bar{gr}({\mathcal{D}})=2$. Observe on the other hand that $\inf_{X}g=0$, but there exists $N$ such that $\inf_{\mathcal{E}_{N}}g\geq 0.26$ (any $N>\max(N_{0}\cup N_{1}\cup N_{2})$ will do the job). Thus, the hypothesis of Corollary \ref{coro} are  fulfilled since $1-(|A|-1)\varepsilon=0.48<1/2$, and existence holds.}
\end{Ex}

So far, we have focussed on the existence issue. However, \cite{bramson/kalikow/1993} proved that even regular $g$-measures (continuous $g$-measures satisfying {\bf (H1')}) might have several  $g$-measures. In view of a result on uniqueness for non-regular $g$-measures, we now give a condition on the set of continuous pasts $X\setminus\mathcal{D}$. To do so, we use the notion of context tree defined below. 
\begin{Def}A \emph{context tree} $\tau$ on $A$ is a subset of $\cup_{k\ge0}A^{\{-k,\ldots,0\}}\cup X$ such that for any $x\in X$, there exists a unique element $v\in\tau$ satisfying $a_{-|v|+1}^{0}=v_{-|v|+1}^{0}$.
\end{Def}
For any $g$-function $g$, we denote by $\tau^{g}$ the smallest  context tree  containing $\mathcal{D}$, called the \emph{skeleton} of $g$.
For instance, coming back to example \ref{comb}, $\tau^{\tilde{g}}=\cup_{i\ge0}\{10^{i}\}\cup\{0_{-\infty}^{0}\}$ and is represented on Figure \ref{fig:peigne}. It is also the skeleton of any $g$-function having only $0_{-\infty}^{0}$ as discontinuity point, such as the $g$-function introduced in Example \ref{exITA}. Pictorially, any $g$-function can be represented as a set of transition probabilities associated to each leaf of the complete tree $A^{-\mathbb{N}}$ and $\tau^{g}$ is the smallest subtree of $A^{-\mathbb{N}}$ which contains $\mathcal{D}$, such that every node has either $|A|$ or $0$ sons. On Figure \ref{fig:arvoregeral} is drawn the (upper part of) the complete tree corresponding to some function $g$ having complicated sets $\mathcal{D}$ and $\tau^{g}$.   \\
\begin{figure}
\centering
\includegraphics[scale=0.9]{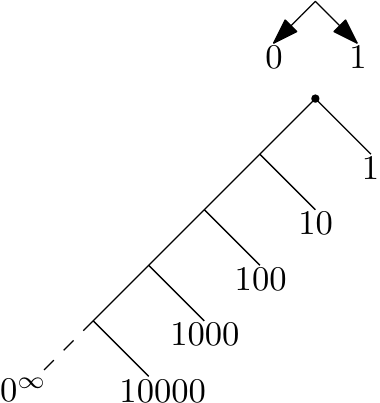}
\caption{The skeleton of the function $\tilde{g}$.}
\label{fig:peigne}
\end{figure}
\begin{figure}
\centering
\includegraphics[scale=1.5]{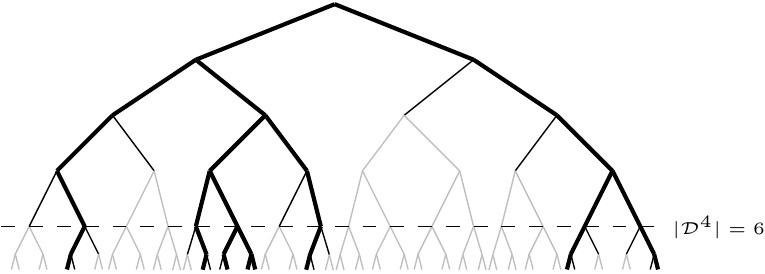}
\caption{An example of set $\mathcal{D}$ (bold black line) for some $g$-function $g$. The black lines represent the context tree $\tau$ corresponding to $\mathcal{D}$ (skeleton of $g$), and the grey lines represent the remaining complete tree. The branches that are not bold black are continuous points for $g$. We can see that $|\mathcal{D}^{1}|=2$, $|\mathcal{D}^{2}|=3$, $|\mathcal{D}^{3}|=4$, $|\mathcal{D}^{4}|=6$, $|\mathcal{D}^{5}|=7$, $|\mathcal{D}^{6}|=8$, ...}
\label{fig:arvoregeral}
\end{figure}

Let us introduce the $n$-variation of a point $x\in X$ that quantifies the rate of continuity of $g$ as
\[
\textrm{var}_{n}(x):=\sup_{y_{-n}^0=x_{-n}^0}|g(y)-g(x)|.
\]
Notice that $\textrm{var}_{n}(x)$ converges to $0$ if and only if $x$ is a continuity point for $g$. As $\textrm{var}_{n}(x)$ actually only depends on $x_{-n}^{0}$, the notation $\textrm{var}_{n}(x_{-n}^{0})$ will sometimes be used.
Now, observe that the set of continuous pasts of a given $g$-function $g$ is the set of pasts $x_{-\infty}^{0}$ such that there exists $v\in\tau^{g}$, $|v|<+\infty$ with $x_{-|v|+1}^{0}=v_{-|v|+1}^{0}$. In particular,  for any $v\in \tau^{g}$ with $|v|<+\infty$,
\[
\textrm{var}^{v}_{n}:=\sup_{x,x_{-|v|+1}^{0}=v}\textrm{var}_{n}(x)\stackrel{n\rightarrow+\infty}{\longrightarrow}0.
\]
For any $v\in\tau^{g}$, $|v|<+\infty$, let $R_{v}:=\sum_{n\ge|v|}[\textrm{var}^{v}_{n}]^2$.
Our assumption on the set of continuous pasts $X\setminus\mathcal{D}$ is 
\begin{equation*}
{\bf (H4)}\,\,\,\sum_{v\in\tau^{g},|v|<\infty}\mu(v)R_{v}<+\infty.
\end{equation*}
Observe that {\bf (H4)} implies that $R_{v}<+\infty$ for any $v\in\tau$.

\begin{Theo}\label{theo:ExistsUniqueMixing}
Suppose that we are given a $g$-function $g$ satisfying {\bf (H1)}, {\bf (H2)} and {\bf (H4)}, then there exists a unique  $g$-measure for $g$.
\end{Theo}

\begin{Rem}
{In this theorem, hypothesis {\bf (H1)} and {\bf (H2)} are mainly used to get the existence of a $g$-measure. Therefore, thanks to Corollary \ref{coro}, the same conclusion holds either assuming {\bf (H1')}, {\bf (H2')} and {\bf (H4)} or {\bf (H1)}, {\bf (H2')}, {\bf (H3)} and {\bf (H4)}.}
\end{Rem}

This result is to be compared to the results of \cite{JO}, which state, in particular, that uniqueness holds when $\textrm{var}_{n}:=\sup_{x_{-n}^{0}}\textrm{var}_{n}(x_{-n+1}^{0})$ is in $\ell^{2}$. In fact, this is mainly what is assumed here, but only on the set of continuous pasts, which has full $\mu$-measure. This is formalised through the more complex hypothesis {\bf (H4)}. We now come back to Examples \ref{exITA} and \ref{ex:3lettres} in order to illustrate Theorem \ref{theo:ExistsUniqueMixing}.

\begin{Exe}[Continued]
\emph{
In this example, we have as skeleton $\tau^{g}=0_{-\infty}^{0}\cup_{i\geq 0}\{10^{i}\}$, so that any $v\in\tau^{g}$ with $|v|=k<\infty$ writes $v=10^{k-1}$ and simple calculations yield, for any $n\geq k$
\[
\textrm{var}^{v}_{n}=(1-2\epsilon)\sum_{i\geq n-k+1}q^{k}_{i}.
\]
Hypothesis {\bf (H4)} is satisfied as soon as
\[
\sum_{k\geq 1}(1-\epsilon)^{k}\sum_{n\geq k+1}\left[\sum_{i\geq n-k+1}q^{k}_{i}\right]^{2}<+\infty.
\]
For instance, if for any $k\geq 1$, $(q^{k}_{i})_{i\ge1}$ is the geometric distribution with parameter $\alpha^{k}$, where $1-\epsilon<\alpha<1$, then
\[
\sum_{k\geq 1}(1-\epsilon)^{k}\sum_{n\geq 1}\left[\sum_{i\geq n+1}q^{k}_{i}\right]^{2}\leq \sum_{k\geq 1}[(1-\epsilon)\alpha^{-1}]^{k},
\]
which is summable. So we have uniqueness for this kernel. 
}
\end{Exe}
\begin{Exem}[Continued]
\emph{
The skeleton of $g$ is 
$$
\tau^{g}=\{0,2\}^{-\mathbb{N}}\cup\{1\}\cup \bigcup_{i\geq 0}\bigcup_{x_{-i}^0\in\{0,2\}^{i+1}}\{1x_{-i}^{0}\}
$$
and for any $v\in\tau^{g}$, $|v|<\infty$, 
\[
\textrm{var}^{v}_{n}\leq2\sum_{i\geq n-|v|}\theta_{i}\,,\,\,\forall n>|v|.
%
%\sup_{a\in A}\sup_{x_{-n+1}^0=y_{-n+1}^0}|g(y\,v\,a)-g(x\,v\,a)|\stackrel{n\rightarrow+\infty}{\longrightarrow}0
\] 
Since this upper bound does not depend on the length of the string $|v|$, it follows that Hypothesis {\bf (H4)} is  satisfied if $\sum_{n\geq 1}\left[\sum_{i\geq n}\theta_{i}\right]^{2}<+\infty$.}
\end{Exem}

%%%%%%%%%%%%%%%%%%%%%%%%%%%%%%%%%%%%%%%%%%%%%%%%%%%%%%%%%%%%%%%%%%%%%%%%%%%%%%%%%%%%%%%%%%%%%%%%%%%%%%%%%%
\section{Proof of Theorem \ref{existence}}
%%%%%%%%%%%%%%%%%%%%%%%%%%%%%%%%%%%%%%%%%%%%%%%%%%%%%%%%%%%%%%%%%%%%%%%%%%%%%%%%%%%%%%%%%%%%%%%%%%%%%%%%%%

Let us define the Perron Frobenius operator $L$ acting on measurable functions $f$ as follows:
$$
Lf(x)=\sum_{a\in A}g(xa)f(xa)=\sum_{x=T(y)}g(y)f(y)
$$
For $\mu\in\M$, let $L^*$ denote the dual operator, that is 
$$
L^*\mu(f)=\mu(Lf)
$$
for any $f\in\B$. The relation between $L^*$ and the $g$-measures is enlightened by the following result.

\begin{Prop}\label{Ledrappier}(\cite{Ledrappier})
$\mu$ is a g-measure if and only if $\mu$ is a probability measure and $L^*\mu=\mu$
\end{Prop}

In view of Proposition \ref{Ledrappier}, the strategy of the proof will be to find a fixed point for $L^*$.
When $g$ is a continuous function, the operator $L$ acts on $\C$ and $L^*$ acts on $\M$, the existence of a g-measure $\mu$ is then a straightforward consequence of the Schauder-Tychonoff theorem.

If $g$ is not continuous, $L$ does not preserve the set of continuous functions. More precisely, if $\D$ is the set of discontinuities of $g$ and $f$ is continuous, then the set of discontinuities of $Lf$ is $T\D$. Still, as $g$ is bounded, $L$ acts on the space $\B$ of bounded functions. More precisely $\Vert Lf\Vert\le\Vert f\Vert$, where $\Vert\ .\ \Vert$ is the uniform norm.
Therefore $L^*$ acts on $\B'$, the topological dual space of $\B$ i.e.
$$
L^*\alpha(f)=\alpha(Lf)
$$
for all $\alpha\in \B'$ and $f\in \B$.

Firstly, the existence of a fixed point $\Lambda\in{\B}'$ for $L^*$ will be proved. Then the hypothesis  {\bf (H1)} and {\bf (H2)} will be shown to imply $\mu({\D})=0$ and $\mu(T{\D})=0$, where $\mu$ is the restriction of $\Lambda$ to the continuous functions. Finally, we will use these two equalities to prove that $\mu$ is indeed a $g$-measure.

\begin{Prop}
There exists a positive functional $\Lambda\in \B'$ with $\Lambda(\1)=1$ such that $L^*\Lambda=\Lambda$.
\end{Prop}

\begin{proof}

Consider the following subset $C$ of $\B'$
$$
C=\{\alpha\in \B', \alpha(\1)=1\ \ \mbox{and}\ \ \alpha(f)\geq 0\ \ \mbox{for all}\ \ f\geq 0\}.
$$
We consider the weak star topology on $\B'$ and $C$. In order to apply Schauder-Tychonoff theorem (\cite{Dunford-Schwartz} V.10.5), it is needed that $L^*$ is well defined and continuous for the weak star topology, that $C$ is compact for this topology, non empty and convex (the two last properties are straightforward). The continuity of $L^*$ is given by a simplification of the proof in \cite{BPS}. The compactness of $C$ follows from Banach-Alaoglu theorem (\cite{Dunford-Schwartz} V.4.2), as $C$ is a closed subset of the unit ball of $\B'$.
\end{proof}

Since $\Lambda_{\mid \C}$ is a positive linear form on $\C$, the Riesz
representation theorem implies that there exists $\mu$, a positive Borel
measure, such that:
$$\forall f\in \C:\ \ \Lambda(f)=\mu(f).$$
In particular, $\mu(\1)=\Lambda(\1)=1$ so that
$\mu$ is a probability measure.

For all $f\in \C$,
$\Lambda(Lf)=\Lambda(f)=\mu(f)$. But $Lf$ is not
necessarily continuous at points
of $T\D$. Notice though that if $f\in\C$ and $Lf\in\C$ then $\mu(f)=\mu(Lf)$. What remains to prove is that this  is true for any $f\in\C$. 

Two more lemmas are needed to go on further in the proof.

\begin{Lem}
$$
P_g(T\D)\le P_g(\D)
$$
\end{Lem}

\begin{proof}
By definition:
 $$
   P_{g}(T\D)=\limsup_{n\to\infty}\frac{1}{n}\log
      \sum_{{B\in\C_{n}}\atop{B\cap T\D\neq\emptyset}}\sup_{B}g_{n}.
 $$
Let $B\in\C_{n}$ such that $B\cap T\D\neq\emptyset$, there
exists $C\in \C_{n+1}$ such
that $C\cap\D\neq\emptyset$. More precisely, there exists $a\in A$ such that $C=C_1(a)\cap T_a^{-1}(B)$.
Moreover, let $x\in B$, then $T_a^{-1}(x)\in\E_n$ and
 $$
g_{n}(x)\le\frac{g_{n+1}(T_{a}^{-1}(x))}{g(T_{a}^{-1}(x))}\le\frac{1}{\inf_{\E_n}
g}\sup_{C}g_{n+1}.
 $$
Since $\E_{n+1}\subset \E_n$,  $\sup_{B}g_{n}\le\frac{1}{\inf_{\E_N} g}\sup_{C}g_{n+1}$ for $n\geq N$. Recall that $\inf_{\E_N} g>0$ by
hypothesis {\bf(H1)}. It comes, for $n\geq N$:
 $$
   \sum_{{B\in\C_{n}}\atop{B\cap T\D\neq\emptyset}}\sup_{B}g_{n}\le\frac{1}{\inf_{\E_N} g}
\sum_{{C\in\C_{n+1}}\atop{C\cap\D\neq\emptyset}}\sup_{C}g_{n+1}
 $$
and thus:
\begin{eqnarray*}
\limsup_{n\to\infty}\frac{1}{n}\log\sum_{{B\in\C_{n}}\atop{
B\cap T\D\neq\emptyset}}\sup_{B}g_{n} \le
\lim_{n\to\infty}\frac{1}{n}\log\frac{1}{\inf_{\E_N}
g}+\limsup_{n\to\infty}\frac{1}{n+1}\log
\sum_{{C\in\C_{n+1}}\atop{C\cap\D\neq\emptyset}}\sup_{C}g_{n+1}.
\end{eqnarray*}

\end{proof}

\begin{Lem}\label{regular}
For all borel sets B,
$$
\mu(B)\le\inf\{\Lambda(O),\ O\ \mbox{open}, \,O\supset B\}
$$
\end{Lem}

\begin{proof}
Since $\mu$ is a regular measure (as a Borel measure on a compact set):
 $$
  \mu(B)=\inf\{\mu(O),\ O\ \mbox{open},\ O\supset B\}.
 $$
Let us fix an open set $O$ and show that: $\mu(O)\le\Lambda(O)$, this will
prove the lemma.
Take
$\varepsilon>0$. Using again the regularity of $\mu$, there exists
$K_{\varepsilon}$, a compact
subset of $O$, such that:
 $$
  \mu(O)<\mu(K_{\epsilon})+\varepsilon.
 $$
Let $f_{\varepsilon}:X\to[0,1]$ be continuous and such that:
$$\left\{\begin{array}{l}
f_{\varepsilon}=1\ \ in\ \ K_{\varepsilon}\\
f_{\varepsilon}=0\ \ in\ \ O^{c}\\
f_{\varepsilon}\le 1\ \ in\ \ O\setminus K_{\varepsilon}.
\end{array}
\right.$$
On one hand, $f_{\varepsilon}\le \1_{O}$ so that
 $$
 \mu(f_{\varepsilon})=\Lambda(f_{\varepsilon})\le\Lambda(O)
 \hbox{ and }
 \sup_{\varepsilon>0}\mu(f_{\varepsilon})\le\Lambda(O).
 $$
On the other hand,
$\mu(f_{\varepsilon})\geq\mu(K_{\varepsilon})>\mu(O)-\varepsilon$ so that:
 $$
  \mu(O) < \mu(K_\varepsilon)+ \varepsilon \leq \mu(f_\varepsilon) +
\varepsilon
 $$
and $\mu(O)\leq\sup_{\varepsilon>0}\mu(f_{\varepsilon})\leq\Lambda(O)$.
\end{proof}

Now, we claim the following:

\begin{Lem}\label{zero_measure}
$$\mu(\D)=0\ \ \mbox{and}\ \ \mu(T\D)=0.$$
\end{Lem}

\begin{proof}
The claim will follow from Lemma \ref{regular} if we can find open neighborhoods $V$
of $\D$ and $W$ of $T\D$ with
$\Lambda(V)$ and $\Lambda(W)$ arbitrarily small. Let us write the proof for $\D$. The same scheme will work for $T\D$.

Recall that $\C_n(\D)=\cup\{C\in\C_{n}, C\cap\D\neq\emptyset\}$.
Using the fixed point property of $\Lambda$ and the definition of pressure, we get, for any $\delta>0$, $N(\delta)$
such that, for all $n>N(\delta)$:
 $$
  \Lambda(\C_n(\D))=\Lambda(L^n\1_{\C_n(\D)})
    \le\sum_{C\in \C_n(\D)}\sup_C g_{n}
    \le (e^{P_g(\D)+\delta})^n.
$$
Taking $\delta=-P_{g}(\D)/2$,
which
is positive by the main hypothesis {\bf (H2)}, we get $\lim_{n\to\infty}\Lambda(\C_n(\D))=0$
and for every $n$, $\C_n(\D)$ is an open neighbourhood of $\D$.
\end{proof}

Finally, the proof of the main theorem writes as follows :

\begin{proof}
Fix $f\in C(X)$ non-negative.

Since $\mu$ is regular (as a Borel measure on a compact set) and as $\mu(\D)=\mu(T\D)=0$ (lemma \ref{zero_measure}), for each $\varepsilon>0$, there exist $U_{\varepsilon}$ open neighbourhood of $\D$ and $V_{\varepsilon}$ open neighbourhood of $T\D$ such that $\mu(U_{\varepsilon})<\varepsilon$ and $\mu(V_{\varepsilon})<\varepsilon$. Let $W_{\varepsilon}=U_{\varepsilon}\cap T^{-1}V_{\varepsilon}$. This is also a neighbourhood of $\D$ such that $\mu(W_{\varepsilon})<\varepsilon$. Moreover, as $TW_{\varepsilon}\subset V_{\varepsilon}$, it comes $\mu(TW_{\varepsilon})<\varepsilon$.

Consider now $f_{\varepsilon}$ with compact support in $X\setminus \D$
such that:
$$\left\{\begin{array}{l}
f_{\varepsilon}=f\ \ in\ \ X\setminus W_{\varepsilon}\\
f_{\varepsilon}\le f\ \ in\ \ W_{\varepsilon}.
\end{array}
\right.$$
First, $Lf_{\varepsilon}$ is continuous on $X$. Namely, $f_{\varepsilon}$ is continuous on $X$ so $Lf_{\varepsilon}$ is on $X\setminus T\D$. Now, if $x\in T\D$, it may be easily checked that the potentially discontinuous part of $Lf_{\varepsilon}$ actually vanishes.
This continuity implies $\mu(Lf_{\varepsilon})=\mu(f_{\varepsilon})$ and
\begin{eqnarray*}
\vert\mu(Lf)-\mu(f)\vert
&=&\vert\mu(Lf_{\varepsilon})+\mu(L(f-f_{\varepsilon}))-\mu(f)\vert
\\
&=&\vert\mu(f_{\varepsilon}-f)+\mu(L(f-f_{\varepsilon}))\vert\\
&\le&\vert 2\Vert f\Vert\mu(W_{\varepsilon})+\mu(L(f-f_{\varepsilon}))\vert.
\end{eqnarray*}
We need to show that $\mu(L(f-f_{\varepsilon}))$ is small. By definition of $f_{\varepsilon}$,
$$
L(f-f_{\varepsilon})(x)=\sum_{a\in A}(f-f_{\varepsilon})(ax)g(ax)\1_{W_{\varepsilon}}(ax)
$$
therefore
\begin{eqnarray*}
\mu(L(f-f_{\varepsilon})) &\le& \Vert g\Vert\ \Vert f-f_{\varepsilon}\Vert\sum_{a\in A}\mu(\1_{W_{\varepsilon}}\circ T_a^{-1}) \\
&\le& 2\Vert g\Vert\ \Vert f\Vert \sum_{a\in A}\mu(T_a(W_{\varepsilon})) \\
&\le& (2\Vert g\Vert\ \Vert f\Vert |A|)\mu(TW_{\varepsilon}).
\end{eqnarray*}
Letting $\varepsilon$ go to zero gives $\mu(Lf)=\mu(f)$.
\end{proof}

%%%%%%%%%%%%%%%%%%%%%%%%%%%%%%%%%%%%%%%%%%%%%%%%%%%%%%%%%%%%%%%%%%%%%%%%%%%%%%%%%%%%%%%%%%%%%%%%%%%%%%%%%%%%%%
\section{Proofs of Corollary \ref{coro} and Theorem  \ref{theo:ExistsUniqueMixing}} 
%%%%%%%%%%%%%%%%%%%%%%%%%%%%%%%%%%%%%%%%%%%%%%%%%%%%%%%%%%%%%%%%%%%%%%%%%%%%%%%%%%%%%%%%%%%%%%%%%%%%%%%%%%%%%%

\begin{proof}[Proof of Corollary \ref{coro} using $\{{\bf (H1')},{\bf (H2')}\}$]
 In view of Theorem \ref{existence}, it is enough to show that hypothesis {\bf (H1')} and {\bf (H2')} imply {\bf (H2)}.
Under hypothesis {\bf (H1')}
\begin{equation}\label{eq1}
g_{n}(x)\leq (1-(|A|-1)\varepsilon)^{n}\,\,\textrm{for any}\,\,\,x.
\end{equation}
It follows that
\[
P_{g}(\mathcal{D})\leq  \limsup_{n\rightarrow+\infty}\frac{1}{n}\log |\mathcal{D}^{n}|(1-(|A|-1)\varepsilon)^{n}.
\]
Now, under {\bf (H2')}, there exists $\alpha\in(0,1)$ such that $|\mathcal{D}^{n}|\leq(\frac{1}{1-(|A|-1)\varepsilon})^{n(1-\alpha)}$ for any sufficiently large $n$. Thus,
\begin{align*}
P_{g}(\mathcal{D})&\leq\limsup_{n\rightarrow+\infty}\frac{1}{n}\log (1-(|A|-1)\varepsilon)^{-n(1-\alpha)}(1-(|A|-1)\varepsilon)^{n}\\&=\limsup_{n\rightarrow+\infty}\frac{1}{n}\log (1-(|A|-1)\varepsilon)^{n\alpha}\\&=\alpha\log (1-(|A|-1)\varepsilon)<0.
\end{align*}
\end{proof}
\begin{proof}[Proof of Corollary \ref{coro} using $\{{\bf (H1)},{\bf (H2')},{\bf (H3)}\}$]
 In view of Theorem \ref{existence}, it is enough to show that hypothesis  {\bf (H1)}, {\bf (H2')} and {\bf (H3)} imply {\bf (H2)}.
Under  {\bf (H1)}
\begin{align*}
\forall n\geq N+1, \forall x\in\E_n, g(x)\le 1-(\vert A\vert-1)\varepsilon.
\end{align*}
Take $B\in\C_n(\D)$ and $x\in B$. Hypothesis{\bf (H3)} implies that $T^ix\in\E_{n-1-i}\subset\E_N$ for all $i\in\{1,\ldots,n-N-1\}$.
Therefore the identity $g_n(x)=g_{n-N}(x)g_N(T^{n-N}x)$ entails for $n\geq N+1$
\begin{equation}\label{eq2}
\forall B\in\C_n(\D), \forall x\in B, g_{n}(x)\leq (1-(|A|-1)\varepsilon)^{n-N}.
\end{equation}
It follows that
\[
P_{g}(\mathcal{D})\leq  \limsup_{n\rightarrow+\infty}\frac{1}{n}\log |\mathcal{D}^{n}|(1-(|A|-1)\varepsilon)^{n-N}.
\]
The rest of the proof runs as before, using hypothesis ${\bf (H2')}$.
\end{proof}

\begin{proof}[Proof of Theorem \ref{theo:ExistsUniqueMixing}]
We already now that existence holds, thanks to hypothesis ${\bf (H1)}$ and ${\bf (H2)}$. 
Remark 2 in \cite{JO} states that,  if for some stationary $\mu$ we have
\[
\int_{X}\mu(dx)\sum_{n}[\textrm{var}_{n}(x)]^{2}<+\infty,
\] 
then $\mu$ is unique. Notice that although \cite{JO} deal with continuous $g$-functions throughout the paper, their uniqueness result only requires existence of a $g$-measure, which is what we have here.

For any point $x\in X$, the sequence  $(\sum_{n=0}^{N}[\textrm{var}_{n}(x)]^{2})_{N\ge0}$ is monotonically increasing and positive, therefore
\begin{align*}
\int_{X}\mu(dx)\sum_{n}[\textrm{var}_{n}(x)]^{2}&=\lim_{N}\int_{X}\mu(dx)\sum_{n=0}^{N}[\textrm{var}_{n}(x)]^{2}\\
&=\lim_{N}\sum_{n=0}^{N}\int_{X}\mu(dx)[\textrm{var}_{n}(x)]^{2}\\
&=\sum_{n}\int_{X}\mu(dx)[\textrm{var}_{n}(x)]^{2}\\
&=\sum_{n}\sum_{x_{-n}^{0}\in A^{n+1}}\mu(x_{-n}^{0})[\textrm{var}_{n}(x_{-n}^{0})]^{2}\,,
\end{align*}
where we used the Beppo-Levi Theorem in the first line, and the fact that $\textrm{var}_{n}(x)$ only depends on $x_{-n}^{0}$ in the last line. We now divide into two parts as follows:
\begin{align*}
\int_{X}\mu(dx)\sum_{n}[\textrm{var}_{n}(x)]^{2}&=\sum_{n}\left[\sum_{x_{-n}^{0}\in\mathcal{D}^{n+1}}\mu(x_{-n}^{0})[\textrm{var}_{n}(x_{-n}^{0})]^{2}\right.\\
&\hspace{2cm}+\left.\sum_{x_{-n}^{0}\in A^{n+1}\setminus\mathcal{D}^{n+1}}\mu(x_{-n}^{0})[\textrm{var}_{n}(x_{-n}^{0})]^{2}\right]
\end{align*}
For the first term of the right-hand side of the equality, we majorate $\textrm{var}_n(x_{-n}^0)$ by 1 and we use the fixed point property of the $g$-measure $\mu$ to obtain, for any $\delta >0$, the existence of $N(\delta)$ such that for all $n> N(\delta)$:
\begin{align*}
\sum_{x_{-n}^{0}\in\mathcal{D}^{n+1}}\mu(x_{-n}^{0})&=\mu(\C_{n+1}(\D))=\mu(L^{n+1}\1_{\C_{n+1}(\D)})\\
&\le\sum_{C\in\C_{n+1}(\D)}\sup_Cg_{n+1}\le(e^{P_g(\D)+\delta})^{n+1}.
\end{align*}
Taking $\delta=-P_g(\D)/2$ (which is strictly positive by hypothesis ${\bf (H2)}$) proves
 $$\sum_n\sum_{x_{-n}^{0}\in\mathcal{D}^{n+1}}\mu(x_{-n}^{0})<\infty.$$
It remains to consider the second term. Recall that if $x_{-n}^{0}\in A^{n+1}\setminus\mathcal{D}^{n+1}$ then there exists $v\in\tau^g$ with $|v|\le n$ such that $v$ is a prefix of $x_{-n}^{-1}$ (denoted by $x_{-n}^{-1}\geq v$). It comes, using {\bf (H4)},
\begin{align*}
\sum_{n}\sum_{x_{-n}^{0}\in A^{n+1}\setminus\mathcal{D}^{n+1}}\mu(x_{-n}^{0})[\textrm{var}_{n}(x_{-n}^{0})]^{2}&=\sum_{n}\sum_{v\in\tau^g:|v|\le n+1}\sum_{x_{-n}^{0}\geq v}\mu(x_{-n}^{0})[\textrm{var}_{n}(x_{-n}^{0})]^{2}\\
&=\sum_{n}\sum_{v\in\tau^g:|v|\le n+1}\mu(v)(\textrm{var}_{n}^{v})^{2}\\
&=\sum_{n}\sum_{v:|v|=n}\mu(v)R_{v}\\&=\sum_{v\in\tau^g:|v|<+\infty}\mu(v)R_{v}<+\infty.
\end{align*}

\end{proof}

\section{Questions and perspectives}

Notice that existence is ensured by an assumption on the set of discontinuous pasts, whereas uniqueness  involves a condition on the set of continuous pasts. For continuous chains, \cite{JO} obtained conditions on the continuity rate of the kernel, ensuring uniqueness.  Making the necessary changes in the hypothesis, Theorem \ref{theo:ExistsUniqueMixing} states that for discontinuous kernel, the same kind of conditions can be used but restricted to the set of continuous pasts, when the measure does not charge the discontinuous pasts. 

 Concerning mixing properties,
 it is known (using the results of \cite{Comets} for example) that chains having summable continuity rate enjoy summable $\phi$-mixing rate. It is natural to expect that, like for the problem of uniqueness, the chains we consider will enjoy the same mixing properties under the same assumption, restricted to the set of continuous pasts.

Finally, it is worth mentioning an interesting parallel with the literature of non-Gibbs state. In this literature, there are examples of stationary measures that are  not \emph{almost-Gibbs},  meaning that there exists stationary measures that give positive weight to the set of discontinuities with respect to both past and future. We do not enter into details and refer to \cite{maes/redig/vanMoffaert/leuven/1999} for the definition of this notion. As far as we know, no such example exist in the world of $g$-measures. More precisely, an interesting question is  whether there exist examples of stationary $g$-measures that are not \emph{almost-regular}, or if, on the contrary, $\mu({\D})=0$ is valid for every stationary $g$-measure.\\

\noindent{\bf Ackowledgement} We gratefully acknowledge X. Bressaud for interesting discussions during the Workshop Jorma's Razor II.

%%%%%%%%%%%%%%%%%%%%%%%%%%%%%%%%%%%%%%%%%%%%%
\bibliographystyle{agsm} 
\bibliography{paccaut}
%%%%%%%%%%%%%%%%%%%%%%%%%%%%%%%%%%%%%%%%%%%%

\vskip 10pt
\noindent Sandro Gallo \\
{\sc Instituto de Matem\'atica, Universidade Federal de Rio de Janeiro} \\
{\tt sandro@im.ufrj.br}
\vskip 10pt
\noindent Fr\'ed\'eric Paccaut \\
{\sc Laboratoire Ami\'enois de Math\'ematiques Fondamentales et Appliqu\'ees cnrs umr 7352,} \\
{\sc Universit\'e de Picardie Jules Verne} \\
{\tt frederic.paccaut@u-picardie.fr}

\end{document}